\documentclass[letterpaper,conference]{ieeeconf}

\IEEEoverridecommandlockouts
\usepackage{amsmath}
\usepackage{amssymb}
\usepackage{graphicx}
\usepackage{color}
\usepackage{algorithm}
\usepackage{bm}
\usepackage{tikz}

\newcommand{\Tr}{\mbox{\rm tr}}
\newcommand{\Var}{\mbox{\rm Var}}
\newcommand{\E}{{\mathbb E}}

\DeclareMathOperator{\argmax}{arg\,max}
\DeclareMathOperator{\argmin}{arg\,min}
\DeclareMathOperator{\vect}{vec}

\newtheorem{theorem}{Theorem}[section]
\newtheorem{lemma}[theorem]{Lemma}
\newtheorem{proposition}[theorem]{Proposition}

\newtheorem{assumption}[theorem]{Assumption}

\newtheorem{problem}[theorem]{Problem}
\newtheorem{definition}[theorem]{Definition}

\newtheorem{remark}[theorem]{Remark}

\newcommand{\update}[1]{{\color{black}#1}}
\newcommand{\rev}[1]{{\color{black} #1}}

\newcommand{\giulio}[1]{{\color{black} #1}}

\renewcommand{\footnoterule}{
  \kern -2pt
  \hrule width 0.3\textwidth height .5pt
  \kern 2pt
}

\title{Preserving Privacy of Finite Impulse Response Systems}

\author{
Giulio Bottegal, Farhad~Farokhi, and Iman Shames
\thanks{F. Farokhi and I. Shames are with the University of Melbourne, Australia. G. Bottegal is with TU Eindhoven, The Netherlands. }
\thanks{e-mails: ffarokhi@unimelb.edu.au (F. Farokhi), ishames@unimelb.edu.au (I. Shames), g.bottegal@tue.nl (G. Bottegal)}
\thanks{The work was supported by a McKenzie Fellowship and the Australian Research Council (LP130100605).}
}

\begin{document}
\maketitle

\begin{abstract}
Adding input and output noises for increasing  model identification error of finite impulse response (FIR) systems is considered. This is motivated by the desire to protect the model of the  system as a trade secret by rendering model identification techniques ineffective. Optimal filters for constructing additive noises that maximizes the identification error subject to maintaining the closed-loop performance degradation below a limit are constructed. 
Furthermore, differential privacy is used for designing output noises that preserve the privacy of the model. 
\end{abstract}

\section{Introduction}
Innovative industries invest resources (e.g., money and time for research and development) to construct new systems and to improve the performance of the previously-deployed ones. To generate revenue and offset the cost of research, they ideally want to capitalize on their achievements. This is sometimes done by restricting the use of their ideas through patents or by hiding the features of their systems as trade secrets. 
When opting for trade secrets, reverse engineering techniques can be used by competitors to unravel their secrets. For instance, model identification tools can be utilized to identify a black-box system or to extract the parameters of a gray-box system. The gained information can be then used to reverse the financial gains. This motivates the use of methods that can render reverse-engineering techniques ineffective. Such methods, however, most often degrade the performance of the system. Therefore, a framework for balancing the need for preserving the trade secrets against maintaining the performance of the systems is required. 

In this paper, linear time-invariant discrete-time finite impulse response (FIR) system are considered. Specifically, the idea of adding noises to the input and output for increasing the error of model identification is explored. A bound on closed-loop performance degradation caused by the additive noise is enforced. An optimal filter for constructing the additive input and output noises that maximizes the identification error subject to maintaining the performance degradation below a threshold is constructed. This is done for both known and unknown input sequences. The former is  useful to make the identification difficult for given inputs, such as the optimal experimental design in the model identification literature~\cite{rojas2007robust}.
The latter, which requires statistics of the input, can accommodate the belief of the  designer on the reverse engineering techniques, e.g., a frequently used input for model identification purposes is a sequence of i.i.d.\footnote{i.i.d. stands for independently and identically distributed.} Gaussian noise~\cite{gevers2006personal}. Finally, differential privacy framework is  used for designing output additive noises that make the system identification difficult without any assumptions on the utilized inputs. 

In differential privacy literature, noises are added to the outcome of statistical queries from databases to preserve the privacy of individuals in the database~\cite{Dwork2006}. This framework was more recently used in dynamical systems~\cite{le2014differentially,huang2014cost}. In differential privacy literature, most often, additive Laplace noises are used and the parameters of the noise are selected according to the sensitivity of the outcome to variations in the data (that should be kept private). However, weaker variants of differential privacy can be achieved by additive Gaussian noises. This is advantageous as adding Laplace noise can make the designer's task considerably more difficult (in terms of utilizing the outputs of the system), e.g., optimal state estimation when measurements are corrupted by  Laplace noise results in non-linearities and memory issues~\cite{farokhi2016optimal}.

To the best of our knowledge, the differential privacy has not been explored in the context of preserving the privacy of dynamical systems with the aim of protecting the model as a trade secret. This has been explored thoroughly in one of the sections of the paper. In addition, in this paper, the problem of preserving the privacy of the systems is cast as a concrete optimization problem that balances the need for keeping the privacy with that of the maintaining the performance. 
This provides a different approach to that of differential privacy in which constraints on the performance degradation cannot be enforced directly to optimally balance between privacy and performance. Finally, note that the problem of releasing the dynamical model of a system under privacy constraints was considered in~\cite{le2013privacy}. In this paper, we take a different approach, i.e., we do not release the model of the system. We want to ensure that inferring an exact model relating inputs and outputs is made difficult.

The rest of the paper is organized as follows. The design of optimal additive input and output noise to hinder system identification is studied in Section~\ref{sec:optimalnoise}. Section~\ref{sec:diffprivacy} uses the differential privacy for constructing additive output noises. A numerical example is provided in Section~\ref{sec:example}. Some concluding remarks are presented in Section~\ref{sec:conc}.

\section{Optimal Additive Noise}\label{sec:optimalnoise}
Here, we investigate the use of additive noise to preserve the privacy of the model information assuming that the eavesdropper uses the best linear unbiased estimate. These results are subsequently generalized (to the case where the model of the eavesdropper is not known) when using the differential privacy framework.
\subsection{Problem Formulation} \label{sec:problem}
In this paper, \rev{for sake of simplicity of presentation, linear single-input single-output (SISO) time-invariant discrete-time systems are considered. All the derivations can be extended to multi-input multi-output (MIMO) systems.} The system is described by the following equation 
\begin{equation}
y_t = H(q^{-1})r_t + e_t,
\end{equation}
where $H(q^{-1})$ represents the transfer function of the system, which is driven by the reference input $r_t$. The output $y_t$ is corrupted by additive white Gaussian noise with variance $\sigma^2$, which is represented by $e_t$. Assume that $H(q^{-1})$ can be well-represented by a finite-impulse response (FIR) system of order $n_h$, i.e.,
$H(q^{-1}) = \sum_{k=0}^{n_h-1} h_k q^{-k}$.
Hence, the dynamics of the system is completely characterized by the vector of coefficients $h := [h_0\,\ldots\,h_{n_h-1}]^\top$. \rev{In this paper, we assume null initial conditions (that is $r_t = 0$ for $t\leq 0$), though extension to any initial condition is straightforward due to the linearity of the underlying system.}

Assume that an adversary is interested in inferring on the process relating $r_t$ to $y_t$ by attempting to estimate $h$ from a set of $N$ input/output measurements $\{r_t,\,y_t\}_{t=1}^N$. To complicate the identification process, an additional component (which is not accessible to the adversary) can be added to the input or to the output of the system to lower the identification accuracy. Let $w_t$ capture such an additional component, which changes the model of the system as
\begin{equation}
y_t = H(q^{-1})r_t + e_t+w_t.
\end{equation} 
This term can capture both the additive input and output noise as discussed, in detail, in what follows.

\begin{assumption} \label{assume:1} The malicious entity is unaware of the presence of the additive input or output noise.
\end{assumption}

This assumption is rather conservative. When using the differential privacy framework in the next section, we can avoid such assumptions. Considering a FIR model for the system and in light of Assumption~\ref{assume:1}, the best linear unbiased estimate (BLUE) of $h$ from perspective of the malicious entity is given by the standard least-squares estimate~\cite[Ch. 4]{soderstrom1988system}. Let us introduce the vectors $y := [y_1 \,\ldots\, y_N]^\top$, $e:= [e_1 \,\ldots\, e_N]^\top$, and $w:= [w_1 \,\ldots\, w_N]^\top$. Assuming that the system is at rest prior to the data collection (i.e., $r_t=0$ for all $t\leq 0$) and defining the matrix
\begin{equation*}
R := 
\begin{bmatrix}
r_1 & 0   & 0 & \ldots & 0 \\ 
r_2 & r_1 & 0 & \ldots & 0 \\ 
\vdots & \vdots & \vdots &\ddots & \vdots \\
r_{n_h} & r_{n_h-1} & r_{n_h-2} &\ldots & r_1 \\
\vdots & \vdots & \vdots &\ddots & \vdots \\ 
r_{N} & r_{N-1}  &  \ldots &\ldots &  r_{N-n_h+1}  \end{bmatrix}\hspace{-.04in},
\end{equation*}
it is evident that $y = Rh + w + e.$
The least-squares estimate of $h$ is then given by
\begin{equation} \label{eq:LS_est}
\hat h = (R^\top R)^{-1}R^\top y.
\end{equation}
Note that this estimator is not the true BLUE, which would require the knowledge of the second order statistics of $w_t$. However, it is the best that the malicious entity can do without the knowledge that $w_t$ exists. This estimator is still unbiased because $\E \{ \hat h \}  = \E \{(R^\top R)^{-1}R^\top (Rh + w + e) \}  =  h + (R^\top R)^{-1}R^\top \E \{w + e \} = h$. 
Then, a measure of the accuracy of the estimation of the impulse response is the covariance matrix of $\hat h$~\cite[Ch. 4]{soderstrom1988system}, namely
\begin{equation}
P_h := \E \{(\hat h-h)(\hat h-h)^\top \}.
\end{equation}
The additional input $w_t$ determines the quality of the estimated system $\hat h$ by entering into the expression of the parameter covariance matrix  $P_h$. Intuitively, the higher the power of $w_t$, the higher $P_h$ (and thus the lower the identification accuracy). On the other hand, $w_t$ has an undesired effect on the output power. Therefore, the additive noise is designed to increase the total variance of $\hat h$ (expressed through the trace of $P_h$) while keeping low the contribution of $w_t$ to the variance of $y_t$. \rev{Let $\lambda_y := \E \left[y_t^2|r_t = 0,\,t\in \mathbb Z \right]$ be such contribution. Note that, if $r_t=0$, the output is driven only by the stationary noise processes $e_t$ and $w_t$ and so $\lambda_y$ is constant in $t$}. 

\begin{problem} \label{prob:det} For a given input $r$, find an appropriate additive noise $w_t$ to maximize the identification error $\Tr(P_h)$ while keeping the performance degradation small by guaranteeing $\lambda_y\leq \gamma_1$. 
\end{problem}

\update{In Problem~\ref{prob:det}, $\gamma_1$ is a pre-selected constant that reflects the maximum tolerable output variance, which is a measure of the performance degradation caused by the additive input and output noises. If $\gamma_1$ is very small, the optimal solution is add no noise. In this case, the closed-loop performance is far superior to protecting the model. However, if $\gamma_1$ is too large, the output of the system is drowned in noise and thus the system becomes practically useless.
}

Here, the additive noise is designed for a given sequence of inputs captured by $r$. This might not be generally feasible as, when dealing with causal systems, the additive noise should be designed and employed prior to receiving the entire sequence of inputs. This design methodology is however very useful to make the identification difficult for a given input, such as those in optimal experimental design in the model identification literature~\cite{rojas2007robust}. Alternatively, a distribution for the input signal can be considered. \rev{Furthermore, the length of the experiment $N$ that the malicious entity is collecting to identify the system is also unknown a priori, and shall be treated as a random quantity.} 

\begin{assumption} \label{assum:1} Let $N\in\mathbb{N}$ be a random number distributed according to $\mathbb{P}\{N=\ell\}=p(\ell)$ for some $p:\mathbb{N}\rightarrow[0,1]$ such that $\sum_{\ell\in\mathbb{N}}p(\ell)=1$. For a given $N$, assume that $r\in\mathbb{R}^{N}$ is distributed according to the conditional probability density function $p(\cdot|N)$ such that $\mathbb{P}\{r\in\mathcal{R}|N\}=\int_{r'\in\mathcal{R}}p(r'|N)\mathrm{d}r'$ for all Lebesgue-measurable sets $\mathcal{R}\subseteq\mathbb{R}^N$.
\end{assumption}

\rev{
\begin{remark} In general, the probability density function of the input signals might not be known in advance. In that case, an online or adaptive approach can be used to estimate the statistical properties of the input as more inputs are revealed over time and design (or update the design of) privacy-preserving filters based on the additional gathered information. The result of this paper can serve as a first step in that direction. This is because if rigorous treatment of the problem for known deterministic inputs or random inputs with known probability distributions is not well understood, the analysis of the online approach would not be possible (or straightforward to say the least).
\end{remark}
}

In this case, the identification error $P_h$ which is used as a measure of privacy should be replaced with $\mathbb{E}\{P_h\}$ with the expectation being taken over random variables $r$ and $N$. 
This allows us to generalize the problem of the interest as follows. 

\begin{problem} \label{prob:sto} For given distributions of random variables $N$ and $r$ following Assumption~\ref{assum:1}, find an appropriate additive noise $w_t$ to maximize the identification error $\Tr(\mathbb{E}\{P_h\})$ while keeping the performance degradation small by guaranteeing $\lambda_y\leq \gamma_1$. 
\end{problem}

In this paper, two families of additive noise are considered, namely, additive output noise and additive input noise. In the remainder of this section, these two families are described.

\begin{figure}
\centering
\begin{tabular}{c}
\begin{tikzpicture}[>=stealth]
\node[draw,rectangle,minimum width=2.0cm,minimum height=1.0cm] (H) at (0,0) {$H(q^{-1})$};
\node[draw,circle] (2) at (+3.2,0) {};
\node[draw,circle] (3) at (+2.0,0) {};
\draw[->,thick=black] (-3.5,0) -- (H);
\draw[->,thick=black] (H) -- (3);
\draw[->,thick=black] (3) -- (2);
\draw[->,thick=black] (+3.2,-1) -- (2);
\draw[->,thick=black] (2) -- (4.5,0);
\node[] at (+2.9,+0.2) {\tiny $+$};
\node[] at (+3.0,-0.3) {\tiny $+$};
\node[] at (+1.7,+0.2) {\tiny $+$};
\node[] at (+1.8,-0.3) {\tiny $+$};
\node[] at (-3.4,+0.3) {$r_t$};
\node[] at (+4.4,+0.3) {$y_t$};
\node[] at (+2.9,-0.9) {$e_t$};
\node[] at (-1.9,-1.0) {$v_t$};
\node[] at (+1.7,-0.9) {$w_t$};
\node[draw,rectangle,minimum width=2.0cm,minimum height=1.0cm] (O) at (0,-1.3) {$L(q^{-1})$};
\draw[->,thick=black] (O) -- (2,-1.3) -- (3);
\draw[->,thick=black] (-2,-1.3) -- (O);
\draw[dashed,-] (-2.8,-2) -- (-2.8,+0.7) -- (+3.8,+0.7) -- (+3.8,-2) -- (-2.8,-2);
\end{tikzpicture}
\\
(a)
\\
\begin{tikzpicture}[>=stealth]
\node[draw,rectangle,minimum width=2.0cm,minimum height=1.0cm] (H) at (0,0) {$H(q^{-1})$};
\node[draw,circle] (1) at (-2.0,0) {};
\node[draw,circle] (2) at (+3.2,0) {};
\draw[->,thick=black] (1) -- (H);
\draw[->,thick=black] (H) -- (2);
\draw[->,thick=black] (+3.2,-1) -- (2);
\draw[->,thick=black] (-3.5,0) -- (1);
\draw[->,thick=black] (2) -- (4.5,0);
\node[] at (+2.9,+0.2) {\tiny $+$};
\node[] at (+3.0,-0.3) {\tiny $+$};
\node[] at (-2.3,+0.2) {\tiny $+$};
\node[] at (-2.2,-0.3) {\tiny $+$};
\node[] at (-3.4,+0.3) {$r_t$};
\node[] at (+4.4,+0.3) {$y_t$};
\node[] at (+2.9,-0.9) {$e_t$};
\node[] at (-2.3,-0.9) {$x_t$};
\node[] at (+1.8,-1.1) {$v_t$};
\node[draw,rectangle,minimum width=2.0cm,minimum height=1.0cm] (I) at (0,-1.3) {$L(q^{-1})$};
\draw[->,thick=black] (I) -- (-2,-1.3) -- (1);
\draw[->,thick=black] (+2,-1.3) -- (I);
\draw[dashed,-] (-2.8,-2) -- (-2.8,+0.7) -- (+3.8,+0.7) -- (+3.8,-2) -- (-2.8,-2);
\end{tikzpicture}
\\
(b)
\end{tabular}
\caption{\label{fig:1} The schematic diagram of the closed-loop system with additive output (a) and input (b) noises. The eavesdropper only has access to the signals outside of the dashed box. }\vspace{-.2in}
\end{figure}
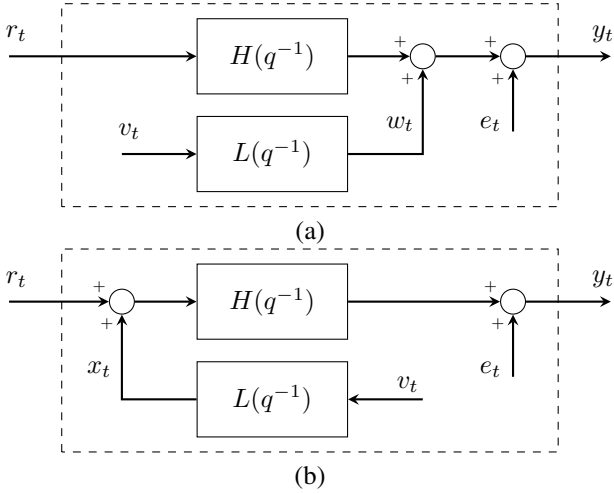

\subsubsection{Additive Output Noise}
\update{Figure~\ref{fig:1} (a) illustrates the schematic diagram of the closed-loop system with additive output noise.}
The additive noise $w_t$ is modelled by a zero-mean moving-average (MA) stochastic process of the form
\begin{equation}
w_t = L(q^{-1}) v_t,
\end{equation}
where $v_t$ is a sequence of i.i.d. zero-mean  noise (\update{which is not necessarily Gaussian}) of unit variance and $L(q^{-1}):=\sum_{k=0}^{n_l}l_kq^{-k}$ is a FIR filter of prescribed order $n_l$. \rev{Then, $w_t$ is a stationary process with zero-mean and well-defined autocovariance function \cite{lindquist2015linear}}. The additive noise $w:=[w_1\,\ldots\, w_N]^\top$ can be expressed as $w=Lv$, where $v := [v_{-n_l+2} \,\ldots v_0 \, v_1 \,\ldots\, v_N]^\top$ and
\begin{equation} \label{eq:def_L}
L:=  \begin{bmatrix} l_{n_l-1} & \ldots & l_0 & 0 & 0 & \ldots & 0 \\ 0 &  l_{n_l-1} & \ldots & l_0 & 0 & \ldots & 0\\ \vdots & \ddots  & \ddots & \ddots & \ddots & \ddots & \vdots \\ 
0 & \ldots & 0 &  l_{n_l-1} & \ldots & l_0 & 0 \\
 0 & 0 & \ldots & 0 & l_{n_l-1} & \ldots & l_0
 \end{bmatrix}.
\end{equation}
The identification error \update{covariance}, in this case, is
\begin{align} \label{eq:ls_variance:outputnoise}
P_h & = (R^\top R)^{-1}R^\top \Var [w+e] R (R^\top R)^{-1} \nonumber \\
	& = (R^\top R)^{-1}R^\top(L L^\top + \sigma^2I_N)R (R^\top R)^{-1}.
\end{align}
Further, the output variance can be determined by
\begin{align}
\lambda_y 
:= \E \{y_t^2|r_t = 0 \}
 = \E \{(w_t + e_t)^2 \}
 = \|l\|^2 + \sigma^2,
\end{align}
where $l=[l_0 \,\ldots\, l_{n_l-1}]^\top$.

\rev{
\begin{remark}
It should be noted that by increasing the order of the noise generation filter $n_l$, the performance can only be improved while maintaining the same privacy guarantee. This is because the optimal solution from the lower order is always feasible in the optimization problem relating to the higher order noise filters. The order of the system is thus only dictated by the available resources   for preserving the privacy of the model.
\end{remark}
}

\subsubsection{Additive Input Noise}
\update{Figure~\ref{fig:1} (b) shows the schematic diagram of the closed-loop system with additive input noise.}
In this case, the additive input noise is denoted by $x_t$ and is modeled by a zero-mean MA stochastic process of the form
\begin{equation}
x_t = L(q^{-1}) v_t,
\end{equation}
where, similarly, $v_t$ is a sequence of i.i.d. zero-mean noise of unit variance and $L(q^{-1})$ is a FIR filter of prescribed order $n_l$ determining the autocorrelation of $x_t$. Then, the new system is described by 
\begin{align}
y_t &= H(q^{-1})(r_t + x_t) + e_t \nonumber \\
&= H(q^{-1})(r_t + L(q^{-1}) v_t) + e_t .
\end{align}
The additive noise $w_t$, in this case, is the contribution of $x_t$ to the output, i.e., $w_t=H(q^{-1})L(q^{-1}) v_t$. Define
\begin{equation}
F(q^{-1}) := H(q^{-1})L(q^{-1}),
\end{equation}
which can be expressed as
\begin{equation}
F(q^{-1}) = \sum_{k=0}^{n_f-1} f_k q^{-k}, \; n_f\hspace{-.03in} =\hspace{-.03in} n_h\hspace{-.03in}+\hspace{-.03in}n_l\hspace{-.03in}-\hspace{-.03in}1.
\end{equation}
Note that $x\hspace{-.03in}:=\hspace{-.03in}[x_1\ldots x_N]^\top$ can be expressed as $x\hspace{-.02in}=\hspace{-.02in}Fv$ with $v := [v_{-n_f+2}\ldots v_0 v_1 \,\ldots v_N]^\top$ and $F$ is defined similarly to $L$ in \eqref{eq:def_L}.
The identification error \update{covariance} becomes
\begin{align} \label{eq:ls_variance}
P_h & = (R^\top R)^{-1}R^\top \Var [w+e] R (R^\top R)^{-1} \nonumber \\
	& = (R^\top R)^{-1}R^\top(F F^\top + \sigma^2I_N)R (R^\top R)^{-1}.
\end{align}
Finally, it can be shown that $\lambda_y = \|f\|^2 + \sigma^2$, where $f=[f_0 \,\ldots\, f_{n_f-1}]^\top$.

\subsection{Deterministic Input} \label{sec:determ}
This part is dedicated to solving Problem~\ref{prob:det}. The results are first presented for the output noise case.
\subsubsection{Additive Output Noise} 
For additive output noise, Problem~\ref{prob:det} can be rewritten as
\begin{subequations} \label{eq:criterion_1:output}
\begin{align} 
\argmax_{l \in \mathbb{R}^{n_l}} & \;\Tr(P_h),  \\
 \mathrm{s.t.} \hspace{.45in} & \;\lambda_y \leq \gamma_1,
\end{align}
\end{subequations}
where $\gamma_1$ denotes the maximum tolerated output variance. Define the performance degradation ratio
\begin{align*}
\rho
:=&\frac{\E \{y_t^2|r_t = 0 \}}{\E \{y_t^2|r_t = 0,w_t=0 \}}
=\frac{\lambda_y}{\sigma^2}.
\end{align*}
If the goal of the designer is to keep the performance degradation ratio below $\epsilon$, the constant $\gamma_1$ can be selected to be smaller than $\sigma^2 \epsilon$. 
\giulio{The following lemma is instrumental to obtain an analytic solution of \eqref{eq:criterion_1:output}.
\begin{lemma} \label{th:dummy_lemma}
Let  
\begin{subequations}\label{eq:E_c_def}
\begin{align} 
E &:=R (R^\top R)^{-1} (R^\top R)^{-1}R^\top \,, \\
c &:=\Tr(\sigma^2(R^\top R)^{-1}) \,,
\end{align}
\end{subequations}
and denote by $Q_l$ a selection matrix such that $\vect(L)=Q_ll$, where $\vect(L)$ is a vector composed of all the columns of the matrix $L$. Then, for the additive noise model, $\Tr(P_h) =  l^\top Q_l^\top ( I_{N+n_l-1} \otimes E )  Q_l l+c $.
\end{lemma}

\begin{proof} See Appendix~\ref{proof:th:dummy_lemma},
\end{proof}

}
Defining $M:=  Q_l^\top ( I_{N+n_l-1} \otimes E)  Q_l$ and noting that the term $c$ is independent of $l$ (and thus can be discarded from the optimization problem), we transform \eqref{eq:criterion_1:output} into
\begin{subequations} \label{eq:optim:eigen:output}
\begin{align} 
\argmax_{l \in \mathbb{R}^{n_l}} & \;l^\top Ml  \\
 \mathrm{s.t.} \hspace{.45in} & \;l^\top l \leq \gamma_1-\sigma^2.
\end{align}
\end{subequations}
The following result can be immediately proved. 

\begin{theorem} \label{tho:eigenvector} The solution of~\eqref{eq:optim:eigen:output} is $l^*=\sqrt{\gamma_1-\sigma^2} \eta^*$, where $\eta^*$ is the normalized eigenvector corresponding to the largest eigenvalue of $M$.
\end{theorem}

\begin{proof} 
The change of variable $\eta=l/\sqrt{\gamma_1-\sigma^2}$ transforms the optimization problem in~\eqref{eq:optim:eigen:output} to
\begin{align*} 
\eta^*\in\argmax_{\eta \in \mathbb{R}^{n_l}} & \;\eta^\top M\eta  \\
 \mathrm{s.t.} \hspace{.45in} & \;\eta^\top \eta \leq 1.
\end{align*}
Note that $M\geq 0$ has at least one positive eigenvalue (as otherwise $M=0$). Therefore, Courant--Fischer--Weyl min-max principle~\cite[p.\,58]{bhatia1997matrix} shows $\eta^*$ is the normalized eigenvector corresponding to the largest eigenvalue of $M$.\end{proof}

\rev{It can be seen that the quality of the model identification drops linearly with increasing $\gamma_1$. At the same time, the performance degradation ratio increases linearly with $\gamma_1$. This capture the trade-off between these two objectives. Note that, for instance, simply increasing the noise variance $\sigma^2$ to the upper bound $\gamma_1$ would determine a linear increase of the identification error, as $P_h$ is proportional to $\sigma^2$. However, this strategy is non-optimal, and Theorem \ref{tho:eigenvector} shows how to obtain the best trade-off between performance degradation and model quality degradation, namely how to get highest linear gain. A comparison between these two strategies is given in Section \ref{sec:example}. }

\rev{If, for a given application, the linear dependency between model quality degradation and system performance degradation is not suitable, one can use the following alternative formulation of the problem:}
\begin{align} \label{eq:criterion_3:output}
\argmin_{l \in \mathbb{R}^{n_l}} \left(\Tr(P_h)\right)^{-1} + \gamma_2 \lambda_y \,,
\end{align}
where $\gamma_2$ determines weight on the performance versus the privacy. This formulation is useful when the constraint on the performance is not hard (i.e., the degradation does not need to be maintained under a given level but large output variations are  not pleasant). This problem is rewritten as
\begin{align} \label{eq:criterion_4:output}
\argmin_{l \in \mathbb{R}^{n_l}} (l^\top M l + c)^{-1} + \gamma_2 \|l\|^2 \,,
\end{align}
where $c$ is defined in \eqref{eq:E_c_def}.

\begin{theorem} \label{tho:weighted} Let $\lambda_1 \geq  \lambda_2 \geq \ldots \geq \lambda_{n_l} \geq 0$ be the eigenvalues of $M$ and $v_1,v_2,\ldots,v_{n_l}$ denote the corresponding eigenvectors. The solution of~\eqref{eq:criterion_4:output} is 
\begin{align*}
l^*=
\begin{cases}
0, & \lambda_{1} \leq \gamma_2 c^2,\\
\sqrt{1/\sqrt{\gamma\lambda_1} - c/\lambda_1}v_{1}, & \mbox{otherwise}.
\end{cases}
\end{align*}
\end{theorem}

\begin{proof}
See Appendix~\ref{proof:tho:weighted}.
\end{proof}


%
\subsubsection{Additive Input Noise}
Similarly, Problem~\ref{prob:det} can be expressed  as
\begin{subequations} \label{eq:criterion_1:input}
\begin{align} 
\argmax_{l \in \mathbb{R}^{n_l}} & \;\Tr(P_h),  \\
 \mathrm{s.t.} \hspace{.45in} & \;\lambda_y \leq \gamma_1.
\end{align}
\end{subequations}
Using the same line of reasoning as in Lemma~\ref{th:dummy_lemma}, we introduce the following instrumental result.
\begin{lemma}
Let $Q_f$ be a selection matrix such that $\vect(F)=Q_ff$. Then, for the additive input noise model,
\begin{equation} \label{eqn:error_to_f}
\Tr(P_h) = f^\top Q_f^\top ( I_{N+n_f-1} \otimes E )  Q_f f+c \,,
\end{equation}
where $E$ and $c$ are defined in \eqref{eq:E_c_def}.
\end{lemma}

\begin{proof} The proof follows the same line of reasoning as in Lemma~\ref{th:dummy_lemma}.\end{proof}

Now, note that the coefficients of the filter $L(q^{-1})$ and filter $F(q^{-1})=H(q^{-1})L(q^{-1})$ are related according to
\begin{equation} \label{eqn:f_h_l}
f = Hl,
\end{equation}
where $H \in \mathbb{R}^{n_f \times n_l}$ is a Toeplitz matrix formed by the coefficients of $h$. Substituting~\eqref{eqn:f_h_l} in~\eqref{eqn:error_to_f} gives
$\Tr(P_h) = l^\top H^\top Q_f^\top ( I_{N+n_f-1} \otimes E )  Q_f H l+c.$
Therefore, the optimization problem in~\eqref{eq:criterion_1:input} can be transformed into 
\begin{subequations} \label{eq:optim:eigen:input}
\begin{align} 
\argmax_{l \in \mathbb{R}^{n_l}} & \;l^\top M' l,  \\
 \mathrm{s.t.} \hspace{.45in} & \;l^\top H^\top H l \leq \gamma_1-\sigma^2,
\end{align}
\end{subequations}
where $M'=H^\top Q_f^\top ( I_{N+n_f-1} \otimes E )  Q_f H $. The following result can be immediately proved. 

\begin{theorem} \label{tho:eigenvector:input} Assume $H^\top H>0$. The solution of~\eqref{eq:optim:eigen:input} is $l^*=\sqrt{\gamma_1-\sigma^2}(H^\top H)^{1/2} \eta^*$, where $\eta^*$ is the normalized eigenvector corresponding to the largest eigenvalue of $(H^\top H)^{-1/2}M'(H^\top H)^{-1/2}$.
\end{theorem}

\begin{proof} Introducing $\eta=(H^\top H)^{-1/2}l/\sqrt{\gamma_1-\sigma^2}$ transforms the optimization problem in~\eqref{eq:optim:eigen:output} to
\begin{align*} 
\eta^*\in\argmax_{\eta \in \mathbb{R}^{n_l}} & \;\eta^\top (H^\top H)^{-1/2}M'(H^\top H)^{-1/2}\eta  \\
 \mathrm{s.t.} \hspace{.45in} & \;\eta^\top \eta \leq 1.
\end{align*}
The rest of the proof follows the same line of reasoning as in the proof of Theorem~\ref{tho:eigenvector}. 
\end{proof}

The condition $H^\top H>0$ is satisfied so long as $H$ has full column rank. This is guaranteed if $h_{n_h}\neq 0$, i.e., no fewer than $n_h$ parameters are required for describing filter $H(q^{-1})$.

\begin{remark} The derivations of this section hold for arbitrary noise distributions as only the first and the second moments of the noise were considered. \update{However, the choice of the Gaussian noise is highly preferred as it makes the integration of the closed-loop system with other control loops much easier. This is an important feature as, most often, off-the-shelf systems are interconnected to achieve complex tasks. Other noise distributions do not lend themselves that easily to integration as they might violate assumptions in the design of the control loops (e.g., Laplace noise results in an increased false alarm rate for fault detection schemes).}
\end{remark}

\rev{
\subsection{Extension to regularized least-squares} \label{sec:RLS}
We now modify the proposed privacy-preserving technique to cope with regularized least-squares estimators. The cost function associated with this type of estimators is 
\begin{equation} \label{eq:cost_RLS}
J_{\mathrm{RLS}}(h) = \|y - Rh\|_2^2 + \eta\|h\|_{K^{-1}}^2 \,,
\end{equation}
where $K$ is a positive semidefinite matrix (usually called a kernel) inducing desired properties in the estimates $\hat h$, see \cite{pillonetto2014kernel} for details on regularized methods for system identification. The solution to 
\eqref{eq:cost_RLS} is 
\begin{equation} \label{eq:sol_RLS}
\hat h = (R^\top R + \eta K^{-1})^{-1} R^\top y = Cy \,,
\end{equation}
with obvious defintion of $C$. This solution is biased. Further, it can be verified (see, e.g., \cite{pillonetto2014kernel}) that the mean square error (MSE) of the estimate is given by
\begin{align} \label{eq:MSE_RLS}
\mathrm{MSE}  = &\E\{(h - \hat h)(h - \hat h)^\top\} \\
=& (I_{n_h} - CR)h h^\top(I_{n_h} - CR)^\top \nonumber\\
&+ CLL^\top C^\top + \sigma^2CC^\top  \nonumber,
\end{align}
the first term on the right hand side corresponding to the bias induced by the regularization penalty. Then, the results of Theorems \ref{tho:eigenvector} and \ref{tho:weighted} hold by redefining
\begin{subequations}\label{eq:E_c_new_def}
\begin{align} 
E &:=C^\top C \,, \\
c &:=\Tr((I_{n_h} - CR)h h^\top(I_{n_h} - CR)^\top + \sigma^2CC^\top) \,,
\end{align}
\end{subequations}
and, accordingly, updating the definition of matrix $M$. Note that the identification performance depends on the parameter $\eta$, regulating the bias-variance trade off, and on the kernel matrix $K$. These are user choices, which are not accessible to privacy-preserving device. One possible way to circumvent this issue is to consider the best possible choice of kernel, which is given by $K = h h^\top$ \cite{pillonetto2014kernel}.
}

\subsection{Random Inputs} \label{sec:random}
The problem of designing an additive output noise is only considered in this section. The results can be easily extended to the design of input noises following the same line of reasoning. Problem~\ref{prob:sto} can be cast as
\begin{subequations} \label{eqn:optim:sto}
\begin{align} 
\argmax_{l \in \mathbb{R}^{n_l}} & \;\Tr(\mathbb{E}\{P_h\})  \\
 \mathrm{s.t.} \hspace{.45in} & \;\lambda_y \leq \gamma_1.
\end{align}
\end{subequations}
Note that $\Tr(P_h)  =\mathbb{E}\{c(r,N)\}+ l^\top \mathbb{E}\{Q_l(N)^\top ( I_{N+n_f-1} \otimes E(r,N) )   Q_l(N) \}l$.
Although having the same definition, $Q_l(N)$, $E(r,N)$, $c(r,N)$ are used instead of $Q_l$, $E$, and $c$ to emphasize they are functions of random variables $N$ and $r$. Define $M'':=\mathbb{E}\{Q_l(N)^\top ( I_{N+n_f-1} \otimes E(r,N) )  Q_l(N) \}$. 
The optimization problem in~\eqref{eqn:optim:sto} can be rewritten as
\begin{subequations} \label{eq:optim:eigen:1}
\begin{align} 
\argmax_{l \in \mathbb{R}^{n_l}} & \;l^\top M'' l,  \\
 \mathrm{s.t.} \hspace{.45in} & \;l^\top l \leq \gamma_1-\sigma^2.
\end{align}
\end{subequations}

\begin{theorem}\label{tho:random} The solution of~\eqref{eq:optim:eigen:1} is $l^*=\sqrt{\gamma_1-\sigma^2} \eta^*$, where $\eta^*$ is the normalized eigenvector corresponding to the largest eigenvalue of $M''$.
\end{theorem}

\begin{proof} The proof follows the same line of reasoning as in Theorem~\ref{tho:eigenvector}. 
\end{proof}

Unfortunately, calculating $M''$ in an explicit from as a function of the distributions of $N$ and $r$ is generally difficult. The following remark provides a numerical algorithm for constructing an approximation of this matrix. 

\begin{remark}[Monte Carlo Simulation] \label{remark:1} Samples of possible input length $N^i$, $i\in\{1,\dots,\theta\}$, are selected randomly. For each $N^i$, $\vartheta$ samples of the inputs of length $N^i$ can be selected. Let these samples be denoted by $r^{ij}$. Define $\hat M''=(1/(\theta\vartheta))\sum_{i=1}^\theta \sum_{j=1}^{\vartheta}
Q_l(N^i)^\top ( I_{N^i+n_f-1} \otimes  E(r^{ij},N^i) )  Q_l(N^i).$
Evidently, $\mathbb{P}\{\|\hat M''-M''\|\geq \epsilon\}\rightarrow 0$ as both $\theta$ and $\vartheta$ tend to infinity for all $\epsilon>0$. Therefore, by selecting enough samples, an arbitrarily close approximation of $M''$ with a high probability can be constructed.
\end{remark}

\section{Relationship to Differential Privacy} \label{sec:diffprivacy}
Throughout this section, the design of an additive output noise is only considered. The results for the additive input noise can be constructed similarly. Furthermore, $h$ is assumed to belong to a compact set $\mathcal{H}\subseteq \mathbb{R}^{n_h}$.

\begin{definition} The system is $\epsilon$-differential private if $\mathbb{P}\{y\in\mathcal{Y}|h\} \leq \exp(\epsilon) \mathbb{P}\{y\in\mathcal{Y}|h'\}$
for all Lebesgue-measurable sets $\mathcal{Y}\subseteq\mathbb{R}$ and $h,h'\in\mathcal{H}$ that differ in at most only one entry, i.e., $\|h-h'\|_0\leq 1$. The system is $(\epsilon,\delta)$-differential private if
$\mathbb{P}\{y\in\mathcal{Y}|h\}
\leq \exp(\epsilon) \mathbb{P}\{y\in\mathcal{Y}|h'\}+\delta$.
\end{definition}

Note that a random variable $w$ is said to follow the Laplace distribution with mean $\mu$ and (scaling) parameter $b>0$ if 
$\mathbb{P}\{w\in\mathcal{W}\}=\int_{w\in\mathcal{W}} (2b)^{-1} \exp(-|w-\mu|/b)\mathrm{d}w$ for all Lebesgue-measurable sets $\mathcal{W}\subseteq \mathbb{R}$.

\begin{theorem} \label{tho:diffprivacy} Assume $w_t$ is i.i.d. Laplace random variables with $b\geq \sup_{h,h'\in\mathcal{H}:\|h-h'\|_0\leq 1} \|Rh-Rh'\|_1/\epsilon$. Then, the system is $\epsilon$-differential private.
\end{theorem}

\begin{proof}
See Appendix~\ref{proof:tho:diffprivacy}.
\end{proof}


Note that $\sup_{h,h'\in\mathcal{H}:\|h-h'\|_0\leq 1} \|Rh-Rh'\|_1$ exists and is finite because $\mathcal{H}$ is assumed to be a compact set.

\begin{theorem} \label{tho:usefulness} Assume $w_t$ is i.i.d. Laplace random variables with scaling parameter  $b$. Then, $\lambda_y=2b^2 + \sigma^2$.
\end{theorem}

\begin{proof} The proof follows from that
$\lambda_y := \E \{y_t^2|r_t = 0 \} = \E \{w_t^2\} +\E\{e_t^2\}= 2b^2 + \sigma^2.$
\end{proof}

Combination of Theorems~\ref{tho:diffprivacy} and~\ref{tho:usefulness} 
 illustrates the trade-off between preserving  privacy and closed-loop performance  because as $\epsilon$ tends to zero (to achieve a higher level of privacy), the performance degrades (i.e., $\lambda_y$ goes to infinity).

\begin{proposition} \label{prop:b} Let $\mathcal{H}:=\{h\in\mathbb{R}^{n_h}\,|\,\underline{h}\leq h_i\leq \overline{h},\forall i\}$. Then, $\sup_{h,h'\in\mathcal{H}:\|h-h'\|_0\leq 1} \|Rh-Rh'\|_1\hspace{-.02in}=\hspace{-.02in}(\overline{h}\hspace{-.02in}-\hspace{-.02in}\underline{h})\sum_{k=1}^N |r_k| $.
\end{proposition}

\begin{proof} See Appendix~\ref{proof:prop:b}.\end{proof}

Proposition~\ref{prop:b} illustrates that the parameter of the Laplace noise $b$ should be increased upon admitting larger input sequences. This is because, with larger $N$, there are more data to extract the system parameters and, thus, the employed mechanism needs to be more conservative to avoid leaking the private information. 
Some relaxations of the differential privacy, e.g., $(\epsilon,\delta)$-differential privacy, that lend themselves to using a Gaussian noise,~e.g.,~\cite{ le2014differentially}. Let for any $\epsilon$ and $\delta$ define $\kappa(\epsilon,\delta)=(\mathcal{Q}^{-1}(\delta)+\sqrt{\mathcal{Q}^{-1}(\delta)^2+2\epsilon})/2$ with $\mathcal{Q}^{-1}$ denoting the inverse of $\mathcal{Q}:x\mapsto \int_{x}^\infty 1/\sqrt{2\pi}\exp(-u^2/2)\mathrm{d}u$.

\begin{theorem} Assume $w_t$ is i.i.d. zero-mean Gaussian noise with $\sigma\geq \kappa(\epsilon,\delta)\sup_{h,h'\in\mathcal{H}:\|h-h'\|_0\leq 1} \|Rh-Rh'\|_2/\epsilon$. Then, the system is $(\epsilon,\delta)$-differential private.
\end{theorem}

\begin{proof} The proof is similar to that of Theorem~\ref{tho:diffprivacy} and can be found in~\cite{le2014differentially}.
\end{proof}


\section{Numerical Examples} \label{sec:example}
Consider the discrete-time system $y_t=G(q^{-1})r_t+e_t,$
where
$G(q^{-1})=(q^{-1} - 0.2q^{-2})/(  1 - 0.9 q^{-1} + 0.17q^{-2}).$
Clearly, $G(q^{-1})$ is not a FIR system. This system can be approximated by the FIR filter 
$H(q^{-1})=
q^{-1}+0.7q^{-2}+0.46q^{-3}+0.295q^{-4}
+0.1873q^{-5}+0.1184q^{-6}+0.0747q^{-7}
+0.0471q^{-8}+0.0297q^{-9}.
$
The quality of the approximation is $\|H(q^{-1})-G(q^{-1})\|=0.0507$. \rev{In the following, we consider the deterministic input and the random input cases.
\subsubsection{Deterministic inputs}
We assume that a sequence of $N=200$ input samples is injected by the malicious entity. The sequence is generated by filtering a white noise process through the low-pass filter $W(q^{-1}) = 1/(1 - 0.95q^{-1})$. We set $\sigma^2=1$ and $\gamma_1=2$, so that we are allow to double the variance of the output. First, we consider the least-squares estimator \eqref{eq:LS_est}. We compute the identification error, given by $\Tr(P_h)$, of least-squares equipped with the proposed privacy preserving technique using output additive noise case with $n_l=10$, and the identification error of least-squares without any privacy preserving device. To get a fair comparison, in the latter case the noise variance is equal to the total noise variance of the former case, that is $\Tr(FF')/N + \sigma^2$. The noise filter designed by the privacy preserving device yields  $\Tr(P_h) = 0.25$, while the variance obtained using standard least-squares is $\Tr(P_h) = 0.17$; we have thus obtained an error increase of approximately $50 \%$.

We now consider regularized least-squares estimators, as described in Subsection \ref{sec:RLS}. We employ as regularization kernel the stable spline kernel $K_{i,j} = \beta^{\max(i,j)}$ (see \cite{pillonetto2014kernel}), with $\beta = 0.7$. The trade off parameter $\eta$ is set as $\eta = 0.1$. Using the proposed privacy preserving technique the obtained MSE of the estimated system is $0.17$, while without privacy preservation (and with  the same noise variance) we get a MSE equal to $0.13$. Increasing $\eta$, the privacy preserving device tends to have a milder effect on the MSE, because the regularized least-squares estimator gives higher weight to the prior knowledge, penalizing the information acquired from data. 

} 

\subsubsection{Random inputs}
Assume that the malicious entity injects a sequence of i.i.d. zero-mean unit-variance Gaussian variables of length $N$ chosen with equal probability from $\{10,\dots,20\}$. The approach of Subsection~\ref{sec:random} is considered for constructing an optimal additive output noise with $n_l=5$. In this example, $M''$ is approximated using the method of Remark~\ref{remark:1} with $\theta=100$ and $\vartheta=1000$. Set $\sigma^2=0.1$ and $\gamma_1=0.2$. Therefore, the performance degradation ratio is upper-bounded as $\rho\leq 2$ (indeed the upper bound is tight due to the nature of the optimal solution). The optimal additive input noise, in this case, is driven by the FIR filter $L(q^{-1})=0.1450+0.0799q^{-1}+0.2125q^{-2}
+0.0799q^{-3}+0.1450q^{-4}$.
Using the Monte Carlo simulation, it can be shown that 
$\Tr(\mathbb{E}\{P_h\})/\Tr(\mathbb{E}\{P_h|w_t=0\})\approx 1.9639.$
Therefore, the system identification error has been approximately doubled at the expense of doubling the output variance. From Theorem~\ref{tho:random}, it can be inferred that
$\Tr(\mathbb{E}\{P_h\})/\Tr(\mathbb{E}\{P_h|w_t=0\})=1+(\eta^{*\top} M''\eta^*)/\mathbb{E}\{c(r,N)\}(\gamma_1-\sigma^2).$


\section{Conclusions} \label{sec:conc}
Adding input and output noises for increasing the model identification error was considered. Optimal filters  for constructing additive coloured noises were designed to maximize the identification error while maintaining the closed-performance degradation below a threshold. Differential privacy was also explored for designing output noises that preserve the privacy of the model. 

\bibliographystyle{ieeetr}
\bibliography{ref}

\appendix

\subsection{Proof of Lemma \ref{th:dummy_lemma}} \label{proof:th:dummy_lemma}
We have
$\Tr(P_h)  =\Tr((R^\top R)^{-1}R^\top(L L^\top + \sigma^2I_N)R (R^\top R)^{-1}) = \Tr (L^\top E L)+c$
Now, note that
$\Tr (L^\top E L)  = \vect (L)^\top \vect (EL)  = \vect (L)^\top ( I_{N+n_l-1} \otimes E )\vect (L)  = l^\top Q_l^\top ( I_{N+n_l-1} \otimes E )  Q_l l ,$
where the second step follows from~\cite[Lemma 4.3.1]{horn2012matrix}.

\subsection{Proof of Theorem \ref{tho:weighted}} \label{proof:tho:weighted}
Taking the derivative of the cost function with respect to $l$ results in
$\partial/\partial l \left[(l^\top M l + c)^{-1} + \gamma_2 \|l\|^2\right]= -2M l/(l^\top M l + c)^2 + \gamma_2 l.$
Setting this derivative equal to zero gives
$\left(M - \gamma_2 (l^\top M l + c)^2  I_{n_l} \right) l = 0.$
The candidate solutions for this equation are either $l = 0$ (referred to as the \emph{type-1 solution}) or vectors $l$ that are parallel to $v_i$ with the condition that $\| l \|^2 = 1/\sqrt{\gamma \lambda_i} - c/\lambda_i $ for all $i = 1,\ldots,n_l$ (referred to as the \emph{type-2 solutions}). An eigenvalue $\lambda_i$ may generate a type-2 solution only if $\lambda_i \geq \gamma_2 c^2$ (since otherwise $l$ would have a negative norm, which is not possible). 

Therefore, if $\lambda_{1} < \gamma_2 c^2$, the only solution to \eqref{eq:criterion_4:output} can be the type-1 solution $l = 0$ (as the condition $\lambda_i \geq \gamma_2 c^2$ cannot be satisfied for any $i$ if it cannot be satisfied for the largest eigenvalue $\lambda_1$). 
This is the case if the penalty on the variance of $y$ is too large and no variations can be tolerated. 

If  $\lambda_i = \gamma_2 c^2$, the two types of solution coincide.


We now verify whether type-1 and type-2 solutions  correspond to global minima of the cost function in~\eqref{eq:criterion_4:output}. Let us define $k: = (l^\top M l + c)$, and also denote the $i$-th row of $M$ by $m_i^\top$. Computing the Hessian of the cost function in~\eqref{eq:criterion_4:output} at $l$ yields
$
J(l) = -\frac{2}{k^2}M + \frac{8}{k^3} V(l) +  2\gamma_2I_{n_l} \,,
$
where $V(l)$ is a matrix such that its entry $(h,k)$ is
$V_{hk}(l) = l^\top m_h  m_k^\top l$. Then 
$J(0) = -\frac{2}{c^2}M + 2\gamma_2 I_{n_l} ,
$
which is positive definite only if $\lambda_1 < \gamma_2 c^2$. This observation shows that the type-1 solution $l=0$ is only a minimum when $\lambda_1 < \gamma_2 c^2$. Noting that for the case where $\lambda_1 < \gamma_2 c^2$, $l=0$ is the only stationary point of the cost function, then it is a global minimum. 

We now study type-2 solutions. Let us define $\alpha_i^2 := 1/\sqrt{\gamma_2 \lambda_i} - c/\lambda_i$, so that a candidate type-2 solution can be written $l^* = \alpha_i v_i,\,i=1,\,\ldots,\,n_l$. In what follows, we first assume that $\lambda_1>\lambda_2\geq \lambda_i$. We then relax this assumption at the end of the proof.
For any $k=1,\,\ldots,\,n_l$, we have
$m_k^\top l^* = m_k^\top\alpha_i v_i = \lambda_i \alpha_i v_{i,k}$, 
where $v_{i,k}$ is the $k$-th entry of $v_i$. Consequently
$V_{hk}(l^*) = l^{*T} m_h  m_k^\top l^* = \lambda_i^2 \alpha_i^2 v_{i,h} v_{i,k},$
and, in matrix notation,
$V(l^*) = \lambda_i^2 \alpha_i^2 v_{i} v_{i}^\top.$
Hence, for any of these solutions, we have $J(l^*) = -2/(\alpha_i^2\lambda_i + c)^2M + (8 \alpha_i^2 \lambda_i^2)/(\alpha_i^2\lambda_i + c)^3v_i v_i^\top + 2\gamma_2 I_{n_l} = -2\gamma_2/\lambda_iM + 8\gamma_2 v_i v_i^\top  - c\sqrt{\gamma_2^3}/\sqrt\lambda_i v_i v_i^\top + 2 \gamma_2 I_{n_l}.$
Since $M$ is positive semidefinite, its eigenvectors form an orthonormal basis \cite[p. 229]{horn2012matrix}. Hence, $M$ admits the decomposition
$M = \sum_{j=1}^{n_l} \lambda_j v_j v_j^\top$. 
Consequently, we can write
$J(l^*) = \sum_{j=1}^{n_l} \eta_j v_j v_j^\top + 2\gamma_2 I_{n_l} ,$
where
\begin{equation*}
\eta_j = \left\{ \begin{array}{ll} -2\gamma_2 \lambda_j/\lambda_i & \quad j\neq i \\  
-2\gamma_2 \lambda_j/\lambda_i + 8\gamma_2 - c \sqrt{\gamma_2^3}/\sqrt{\lambda_i} & \quad j= i \,.
\end{array} \right.
\end{equation*}
Due to the orthonormality of the $v_j$, the eigenvalues of $J(l^*)$ are then $\eta_j + 2\gamma_2,\,j=1,\,\ldots,\,n_l$.

Consider now a candidate type-2 solution corresponding to an eigenvalue $\lambda_i,\,i\geq 2$. In this case, one of the eigenvalues of $J(l^*)$ is $2\gamma_2\left(1-\lambda_1/\lambda_i\right)$, which is negative under the assumption $\lambda_1>\lambda_2\geq \lambda_i$. Therefore, all the candidate type-2 solution corresponding to an eigenvalue $\lambda_i,\,i\geq 2$, are not minimums so we must discard them. As for $\lambda_1$, the set of eigenvalues $\rho_j$ of $J(l^*)$ are
\begin{equation*}
\rho_j = 2\gamma_2\left( 1-\lambda_j/\lambda_1\right)+\left\{\begin{array}{ll}
8\gamma_2(1 -c \sqrt{\gamma_2/\lambda_1}), &\; \mbox{if }j=1, \\
0, &\;\mbox{otherwise}
\end{array} \right.
\end{equation*}
which are all positive for $\lambda_1 > c^2\gamma_2$. Therefore, $J(l^*)$ is positive definite for $l^* = \sqrt{1/\sqrt{\gamma_2\lambda_1} - c/\lambda_1}v_{1}$ and, since there are no other minimums, this corresponds to a global minimum.

Now, assume that $\lambda_1=\lambda_2=\cdots=\lambda_j>\lambda_{j-1}$. Following the same steps as the proof above, we can show that none of the type-2 solutions corresponding to $\lambda_i$ with $j-1\leq i\leq n_l$ can be a minimizer (because the Hessian is indefinite for them). Similarly, we can also show that all the type-2 solutions corresponding to $\lambda_i$ with $1\leq i\leq j$ are at least local minimums (because the Hessian is positive definite). To show that these points are also a global minimizer, we need to prove that they have the same cost. Let $l^*_{i_1}= \sqrt{1/\sqrt{\gamma_2\lambda_{i_1}} - c/\lambda_{i_1}}v_{i_1}$ and $l^*_{i_2}= \sqrt{1/\sqrt{\gamma_2\lambda_{i_2}} - c/\lambda_{i_2}}v_{i_2}$ for any $1\leq i_1,i_2\leq j$. We have $({l^*_{i_1}}^\top M l^*_{i_1} + c)^{-1} + \gamma_2 \|l^*_{i_1}\|^2=(\lambda_{i_1} + c)^{-1} + \gamma_2 (1/\sqrt{\gamma_2\lambda_{i_1}} - c/\lambda_{i_1})=(\lambda_{i_2} + c)^{-1} + \gamma_2 (1/\sqrt{\gamma_2\lambda_{i_2}} - c/\lambda_{i_2})=({l^*_{i_2}}^\top M l^*_{i_2} + c)^{-1} + \gamma_2 \|l^*_{i_2}\|^2,$
where the first equality follows from that $\lambda_{i_1}=\lambda_{i_2}$.

\subsection{Proof of Theorem \ref{tho:diffprivacy}} \label{proof:tho:diffprivacy}
It can be proved that
\begin{align}
\mathbb{P}\{y&\in\mathcal{Y}|h,e\}\nonumber\hspace{-.05in}\\ =
&\bigg(\frac{1}{2b}\bigg)^{N}\hspace{-.05in}\int_{\mathbb{R}^N}
\chi(Rh+w+e\in\mathcal{Y})\exp(-\|w\|_1/b)\mathrm{d}w \nonumber\\
=&\bigg(\frac{1}{2b}\bigg)^{N}\hspace{-.05in}\int_{\mathbb{R}^N}
\chi(u\in\mathcal{Y})\exp(-\|u-Rh-e\|_1/b)\mathrm{d}u \nonumber\\
\leq &\exp(\|Rh'-Rh\|_1/b)\nonumber\\
&\times\bigg(\frac{1}{2b}\bigg)^{N}\hspace{-.05in}\int_{\mathbb{R}^N}
\chi(u\in\mathcal{Y})\exp(-\|u-Rh'-e\|_1/b)\mathrm{d}u \nonumber\\
=&\exp(\|Rh'\hspace{-.03in}-\hspace{-.03in}Rh\|_1/b) \mathbb{P}\{\hspace{-.02in}y\hspace{-.03in}\in\hspace{-.03in}\mathcal{Y}|h',e\}, \label{eqn:proof:Laplance}
\end{align}
where $\chi(\cdot)$ is a characteristic function, i.e., $\chi(y\in\mathcal{Y})=1$ if $y\in\mathcal{Y}$ and $\chi(y\in\mathcal{Y})=0$ if $y\notin\mathcal{Y}$, and the inequality follows from 
$\|u-Rh'-e\|_1=\|u-Rh'-e-Rh+Rh\|_1
\leq \|u-Rh-e\|_1+\|Rh'-Rh\|_1.$
Integrating~\eqref{eqn:proof:Laplance} over $e$ gives
$\mathbb{P}\{y\in\mathcal{Y}|h\}
\leq \exp(\|Rh'-Rh\|_1/b) \mathbb{P}\{y\in\mathcal{Y}|h'\}
=\exp(\epsilon)\mathbb{P}\{y\in\mathcal{Y}|h'\}$.

\subsection{Proof of Proposition \ref{prop:b}} \label{proof:prop:b}
If $h,h'$ only differ in entry $j$, $\|Rh-Rh'\|_1=|h_j-h'_j|\sum_{k=1}^{N-j} |r_k|,
$ $1\leq j\leq n_h$. Thus,
$\sup_{\underline{h}\leq h_j,h'_j\leq \overline{h}}\|Rh-Rh'\|_1=(\overline{h}\hspace{-.02in}-\hspace{-.02in}\underline{h})\sum_{k=1}^{N-j} |r_k|.$
The rest of the proof follows from that all the terms in the sum are positive (and setting $j=1$ keeps the most terms).
\end{document}